\documentclass[10pt]{amsart}
\usepackage[cp1251]{inputenc}
\usepackage{amstext}
\usepackage{amsfonts}
\usepackage{amssymb}
\usepackage{amsbsy}
\usepackage{amsmath}
\usepackage{latexsym}
\usepackage{xy}
\usepackage{hhline}
\xyoption{all}
\newcommand{\vtx}[1]{*+[o][F-]{\scriptscriptstyle #1}} 

\newcounter{num}[section] %

\newenvironment{theo}
{\refstepcounter{num}%
\bigskip\noindent{\bf Theorem~\arabic{section}.\arabic{num}. }\it}

\newenvironment{cor}
{\refstepcounter{num}%
\bigskip\noindent{\bf Corollary~\arabic{section}.\arabic{num}. }\it}

\newenvironment{lemma}
{\refstepcounter{num}%
\bigskip\noindent{\bf Lemma~\arabic{section}.\arabic{num}. }\it}

\newenvironment{example}
{\refstepcounter{num}%
\bigskip\noindent{\bf Example~\arabic{section}.\arabic{num}.}}

\newenvironment{conj}
{\refstepcounter{num}%
\bigskip\noindent{\bf Conjecture~\arabic{section}.\arabic{num}. }\it}

\newenvironment{remark}
{\refstepcounter{num}%
\bigskip\noindent{\bf Remark~\arabic{section}.\arabic{num}.}}

\newcommand{\definition}[1]
{\refstepcounter{num}%
\bigskip\noindent{\bf Definition~\arabic{section}.\arabic{num}}~({\it #1}).}

\newcommand{\defin}{\refstepcounter{num}%
\bigskip\noindent{\bf Definition~\arabic{section}.\arabic{num}.}}

\newcommand{\Ref}[1]{(\ref{#1})}

\newcounter{thepic}

\newenvironment{eq}{\begin{equation}}{\end{equation}}

\newcommand{\si}{\sigma}
\newcommand{\al}{\alpha}
\newcommand{\be}{\beta}
\newcommand{\ga}{\gamma}

\newcommand{\de}{\delta}

\newcommand{\ov}[1]{\overline{#1}}
\newcommand{\un}[1]{{\underline{#1}} }
\newcommand{\id}[1]{{{\rm id}\{{#1}\}}}

\newcommand{\tr}{\mathop{\rm tr}}

\newcommand{\mdeg}{\mathop{\rm mdeg}}
\newcommand{\Char}{\mathop{\rm char}}

\newcommand{\pwr}{{\mathop{\rm{pwr }}}}
\newcommand{\eqIII}{\asymp} 

\newcommand{\FF}{{\mathbb{F}}}   
\newcommand{\NN}{{\mathbb{N}}}

\newcommand{\RR}{{\mathbb{R}}}

\newcommand{\M}{\mathcal{M}}  
 
\begin{document}
\renewcommand{\refname}{References}
\thispagestyle{empty}

\title{On the nilpotency degree of the algebra with identity $x^n=0$.}%
\author{{Artem A. Lopatin}}%
\noindent\address{\noindent{}Artem A. Lopatin%
\newline\hphantom{iiii} Omsk Branch of
\newline\hphantom{iiii} Federal State Budgetary Establishment
\newline\hphantom{iiii} Sobolev Institute of Mathematics, SB RAS,
\newline\hphantom{iiii} Pevtsova street, 13,
\newline\hphantom{iiii} 644099, Omsk, Russia%
\newline\hphantom{iiii} http://www.iitam.omsk.net.ru/\~{}lopatin}%
\email{artem\underline{ }lopatin@yahoo.com}%

\vspace{1cm}
\maketitle {\small
\begin{quote}
\noindent{\sc Abstract. } Denote by $C_{n,d}$ the nilpotency degree of a relatively free algebra generated by $d$ elements and satisfying the identity $x^n=0$. Under assumption that the characteristic $p$ of the base field is greater than $n/2$, it is shown that $C_{n,d}<n^{\log_2(3d+2)+1}$ and $C_{n,d}<4\cdot 2^{\frac{n}{2}}d$. In particular, it is  established that the nilpotency degree $C_{n,d}$ has a polynomial growth in case the number of generators $d$ is fixed and $p>\frac{n}{2}$.  For $p\neq2$ the nilpotency degree  $C_{4,d}$ is described with deviation $3$ for all $d$. As an application, a finite generating set for the algebra $R^{GL(n)}$ of $GL(n)$-invariants of $d$ matrices is established in terms of $C_{n,d}$. Several conjectures are formulated.

\medskip


\end{quote}
}

\section{Introduction}\label{section_intro}

We assume that $\FF$ is an infinite field of arbitrary characteristic $p=\Char{\FF}\geq0$. All vector spaces, algebras and modules are over $\FF$ and all algebras are associative with unity unless otherwise stated. 

We denote by $\M=\M(x_1,\ldots,x_d)$ the semigroup (without unity) freely generated by {\it letters}  $x_1,\ldots,x_d$ and denote by $\M_{\FF}=\M_{\FF}(x_1,\ldots,x_d)$ the vector space with the basis $\M$. Let 
$$N_{n,d}=N_{n,d}(x_1,\ldots,x_d)=\frac{\M_{\FF}}{\id{x^n\,|\,x\in\M_{\FF}}}$$%
be the relatively free algebra with the identity $x^n=0$. The connection between this algebra and analogues of the Burnside problems
for associative algebras suggested by Kurosh and Levitzky is discussed in recent survey~\cite{Zelmanov07} by Zelmanov.

We write 
$$C_{n,d}=\min\{c>0\,|\,a_1\cdots a_c=0 \text{ for all }a_1,\ldots,a_c\in N_{n,d}\}$$ 
for the {\it nilpotency} degree of $N_{n,d}$. Since $C_{1,d}=1$ and $C_{n,1}=n$, we assume that $n,d\geq 2$ unless otherwise stated. Obviously, $C_{n,d}$ depends only on $n$, $d$, and $p$.

We consider the following three cases:
\begin{enumerate}
\item[(a)] $p=0$;

\item[(b)] $0<p\leq n$;

\item[(c)] $p>n$.
\end{enumerate}
By the well-known Nagata--Higman Theorem (see~\cite{Nagata53} and~\cite{Higman56}), which at first was proved by Dubnov and Ivanov~\cite{Dubnov43} in 1943, $C_{n,d}< 2^n$ in cases~(a) and~(c).  As it was pointed out in~\cite{DKZ02}, $C_{n,d}\geq d$ in case~(b); in particular, $C_{n,d}\to \infty$ as $d\to \infty$.  Thus, the case~(b) is drastically different from cases~(a) and~(c). In 1974 Razmyslov~\cite{Razmyslov74} proved that $C_{n,d}\leq n^2$ in case~(a).  As about lower bounds on $C_{n,d}$, in 1975 Kuzmin~\cite{Kuzmin75} established that $C_{n,d}\geq \frac{1}{2}n(n+1)$ in cases~(a) and~(c) and conjectured that $C_{n,d}$ is actually equal to $\frac{1}{2}n(n+1)$ in these cases. A proof of the mentioned lower bound was reproduced in books~\cite{Drensky_book04} and~\cite{Belov_book05} (see page~341). Kuzmin's conjecture is still unproven apart from some partial cases. Namely, the conjecture holds for $n=2$ and $n=3$ (for example, see~\cite{Lopatin_Comm1}). In case~(a) the conjecture was proved for $n=4$ by Vaughan--Lee~\cite{Vaughan93} and for $n=5$, $d=2$ by Shestakov and  Zhukavets~\cite{Shestakov04}. 

Using approach by Belov~\cite{Belov92}, Klein~\cite{Klein00} obtained that for an arbitrary characteristic the inequalities $C_{n,d}<\frac{1}{6}n^6 d^n$ and $C_{n,d}<\frac{1}{(m-1)!}n^{n^3} d^{m}$ hold, where $m=[n/2]$. Here $[a]$ (where $a\in\RR$) stands for the largest integer $b<a$. Recently, Belov and Kharitonov~\cite{Belov2011} established that $C_{n,d}\leq	2^{18}\cdot n^{12\log_3(n)+28}d$ (see Remark~\ref{remark_Belov} for more details). Moreover, they proved that a similar estimation also holds for the Shirshov Height of a finitely generated PI-algebra. We can summarize the above mentioned bounds on the nilpotency degree as follows:
\begin{enumerate}
\item[$\bullet$] if $p=0$, then $\frac{1}{2}n(n+1)\leq C_{n,d}\leq n^2$;

\item[$\bullet$] if $0<p\leq n$, then $d\leq C_{n,d}<\frac{1}{6}n^6 d^n$ and $C_{n,d}\leq 2^{18}\cdot n^{12\log_3(n)+28}d$;

\item[$\bullet$] if $p>n$, then $\frac{1}{2}n(n+1)\leq C_{n,d}< 2^n$.
\end{enumerate}
For $d>0$ and arbitrary characteristic of the field the nilpotency degree $C_{n,d}$ is known for $n=2$ (for example, see~\cite{DKZ02}) and $n=3$ (see~\cite{Lopatin_Comm1} and~\cite{Lopatin_Comm2}):
$$
C_{2,d}=\left\{
\begin{array}{rl}
3,&\text{if } p=0 \text{ or }p>2\\
d+1,&\text{if } p=2 \\
\end{array}
\right.
\;\text{ and }\;
C_{3,d}=\left\{
\begin{array}{rl}
6,&\text{if } p=0 \text{ or }p>3\\
6,&\text{if } p=2 \text{ and }d=2\\
d+3,&\text{if } p=2 \text{ and } d>2\\
3d+1,&\text{if }p=3.\\
\end{array}
\right..
$$

In this paper we obtained the following upper bounds on $C_{n,d}$: 
\begin{enumerate}
\item[$\bullet$] $C_{n,d}<n^{\log_2(3d+2)+1}$ in case $p>\frac{n}{2}$ (see Corollary~\ref{cor_poly}). Therefore, we establish a polynomial upper bound on $C_{n,d}$ under assumption that the number of generators $d$ is fixed.  

\item[$\bullet$] $C_{n,d}<4\cdot 2^{\frac{n}{2}}d$ for $\frac{n}{2}<p\leq n$ (see Corollary~\ref{cor1}). Modulo Conjecture~\ref{conj_n2}, we prove that $C_{n,d}<n^2 \ln(n) d$ for $\frac{n}{2}<p\leq n$ (see Corollary~\ref{cor_modulo_conj}). 

\item[$\bullet$]  $C_{4,d}$ is described with deviation $3$ for all $d$ under assumption that $p\neq2$ (see Theorem~\ref{theo_n_is_4}).
\end{enumerate}%
\noindent{}Note that even in the partial case of $p>n$ and $d=2$ a polynomial bound on $C_{n,d}$ has not been known. If $n$ is fixed and $d$ is large enough, then the bound from Corollary~\ref{cor1} is better than that from Corollary~\ref{cor_poly}. In Remark~\ref{remark_Belov} we show that for $p>\frac{n}{2}$, $4\leq n\leq 2000$, and all $d$ the bound from Corollary~\ref{cor1} is at least $10^{20}$ times better than the bounds by Belov and Kharitonov~\cite{Belov2011}.

As an application, we consider the algebra $R^{GL(n)}$ of $GL(n)$-invariants of several matrices and describe a finite generating set for $R^{GL(n)}$ in terms of $C_{n,d}$ (see Theorem~\ref{theo_gen_system}). We conjecture that $R^{GL(n)}$ is actually generated by its elements of degree less or equal to $C_{n,d}$ (see Conjecture~\ref{conj_inv}).

The paper is organized as follows. In Section~\ref{section_main_theo} we establish a key recursive formula for an upper bound on $C_{n,d}$ that holds in case $p=0$ or $p>\frac{n}{2}$ (see Theorem~\ref{theo_main}): 
\begin{eq}\label{eq_main}
C_{n,d}\leq d\sum_{i=2}^n (i-1) C_{\left[n/i\right],d}+1.
\end{eq}%
The main idea of proof of Theorem~\ref{theo_main} is the following one. We introduce some partial order $>$ on $\M$ and the $\eqIII$-equivalence on $\M_{\FF}$ in such a way that $f \eqIII h$ if and only if the image of $f-h$ in $N_{n,d}$ belongs to $\FF$-span of elements that are bigger than $f-h$ with respect to $>$. Since $N_{n,d}$ is homogeneous with respect to degrees, there exists a $w\in\M$ satisfying $w\not\eqIII 0$ and $C_{n,d}=\deg{w}+1$. Thus we can deal with the $\eqIII$-equivalence instead of the equality in $N_{n,d}$. Some relations of $N_{n,d}$ modulo $\eqIII$-equivalence resembles relations of $N_{k,d}$ for $k < n$ (see formula~\Ref{eq0}). This fact allows us to obtain the upper bound on $C_{n,d}$ in terms of $C_{k,d}$, where $k< n$. To illustrate the proof of Theorem~\ref{theo_main}, in Example~\ref{ex_n5} we consider the partial case of $n=5$ and $p\neq 2$. Note that a similar approach to the problem of description of $C_{n,d}$ can be originated from every partial order on $\M$.

In Section~\ref{section_poly} we apply recursive formula~\Ref{eq_main} several times to obtain the polynomial bound from Corollary~\ref{cor_poly}. On the other hand, in Section~\ref{section_cor} we use formula~\Ref{eq_main} together with the Nagata--Higman Theorem to establish Corollary~\ref{cor1}. Formula~\Ref{eq_main} is applied to the partial case of $n\leq9$ in Corollary~\ref{cor1a}.   

In Section~\ref{section_n_is_4} we develop the approach from Section~\ref{section_main_theo} for $n=4$ to prove Theorem~\ref{theo_n_is_4}. We define a new partial order $\succ$ on $\M$, which is weaker than $>$, and obtain a new $\approx$-equivalence on $\M_{\FF}$, which is stronger than $\eqIII$-equivalence. Considering relations of $N_{4,d}$ modulo $\approx$-equivalence, we obtain the required bounds on $C_{4,d}$. 

Section~\ref{section_matrix_inv} is dedicated to the algebras of invariants of several matrices.

We end up this section with the following optimistic conjecture, which follows from Kuzmin's conjecture. We write $C_{n,d,p}$ for $C_{n,d}$.  

\begin{conj}\label{conj_optimistic}
For all $p>n$ we have $C_{n,d,0}=C_{n,d,p}$.
\end{conj}
\bigskip%

\noindent{}This conjecture holds for $n=2,3$ (see above). Note that Conjecture~\ref{conj_n2} follows from Conjecture~\ref{conj_optimistic} by the above mentioned result by Razmyslov.

\section{Recursive upper bound}\label{section_main_theo}

We start with some notations. Let $\NN=\{1,2,\ldots\}$, $\NN_0=\NN\sqcup\{0\}$,  and $\FF^{\ast}=\FF\backslash \{0\}$. Denote $\M_1=\M\sqcup\{1\}$, i.e., we endow $\M$ with the unity. Given a letter $x$, denote by $\M^{\neg x}$ the set of words $a_1\cdots a_r\in\M$ such that neither letter $a_1$ nor letter $a_r$ is equal to $x$ and $r>0$.

For $a\in\M_1$ and a letter $x$ we denote by $\deg_x(a)$ the degree of $a$ in the letter $x$ and by $\mdeg(a)=(\deg_{x_1}(a),\ldots,\deg_{x_r}(a))$ the multidegree of $a$. For short, we write $1^r$ for $(1,\ldots,1)$ ($r$ times) and say that $a$ is {\it multilinear} in case $\mdeg(a)=1^r$.

Given $\un{\al}=(\al_1,\ldots,\al_r)\in\NN_{0}^r$, we set $\#\un{\al}=r$, $|\un{\al}|=\al_1+\cdots+\al_r$, and $\un{\al}^{\rm ord}=(\al_{\si(1)},\ldots,\al_{\si(r)})$ for a permutation $\si\in S_r$ such that $\al_{\si(1)}\geq\cdots\geq\al_{\si(r)}$. If $r=0$, then we say that $\un{\al}$ is an empty vector and write $\un{\al}=\emptyset$. Note that for $\un{\al}=\emptyset$ we also have $\un{\al}^{\rm ord}=\emptyset$.

Given $\un{\theta}\in\NN_{0}^r$ with $|\un{\theta}|=n$ and $a_1,\ldots,a_r\in\M$, denote by $T_{\un{\theta}}(a_1,\ldots,a_r)$ the coefficient of $\al_1^{\theta_1}\cdots \al_r^{\theta_r}$ in $(\al_1 a_1 +\cdots + \al_r a_r)^n$, where $\al_i\in\FF$. Since the field $\FF$ is infinite, standard Vandermonde arguments give that $T_{\un{\theta}}(a_1,\ldots,a_r)=0$ holds in $N_{n,d}$.


\definition{of $\pwr_x(a)$} Let $x$ be a letter and $a=a_1 x^{\al_1}\cdots a_r x^{\al_r} a_{r+1}\in\M$, where $r\geq 0$, $a_1,a_{r+1}\in\M_1$, $a_2,\ldots, a_r\in\M$, $\al_1,\ldots,\al_r>0$, and $\deg_x(a_i)=0$ for all $i$. Then we denote by $\pwr_x(a)=(\al_1,\ldots,\al_r)$ the {\it $x$-power} of $a$. In particular, if $\deg_x(a)=0$, then $\pwr_x(a)=\emptyset$.
\bigskip

Let $\un{\al}\in\NN^r$, $\un{\be}\in\NN^s$ ($r,s\geq0$) satisfy $\un{\al}=\un{\al}^{\rm ord}$ and $\un{\be}=\un{\be}^{\rm ord}$. Then we write $\un{\al}>\un{\be}$ if one of the following conditions holds:
\begin{enumerate}
\item[$\bullet$] $r<s$;

\item[$\bullet$] $r=s$ and $\al_1=\be_1,\ldots,\al_l=\be_l$, $\al_{l+1}>\be_{l+1}$ for some $0\leq l< r$.
\end{enumerate}
As an example, $(2,2,2)<(3,2,1)<(4,1,1)<(3,3)<(4,2)<(5,1)<(6)<\emptyset$.

\defin\label{def_order} Let $x$ be a letter and $a,b\in\M$. Introduce the partial order $>$ and the $\gtrless$-equivalence on $\M$ as follows: 
\begin{enumerate}
\item[$\bullet$] $a>b$ if and only if $\pwr_x(a)^{\rm ord}>\pwr_x(b)^{\rm ord}$ for some letter $x$ and $\pwr_y(a)^{\rm ord}\geq\pwr_y(b)^{\rm ord}$ for every letter $y$;

\item[$\bullet$] $a\gtrless b$ if and only if $\pwr_y(a)^{\rm ord}=\pwr_y(b)^{\rm ord}$ for every letter $y$; in particular, $\mdeg{a}=\mdeg{b}$. 
\end{enumerate}

\begin{remark}\label{remark1}
There is no an infinite chain $a_1<a_2<\cdots$ such that $a_i\in\M$ and $\deg(a_i)=\deg(a_j)$ for all $i,j$.
\end{remark}

\definition{of the $\eqIII$-equivalence}\label{def_eqIII} 
\begin{enumerate}
\item[1.] Let $f=\sum_i\al_i a_i\in\M_{\FF}$, where $\al_i\in\FF^{\ast}$, $a_i\in\M$, and $a_i\gtrless a_{i'}$ for all $i,i'$.  Then $f\eqIII 0$ if $f=0$ in $N_{n,d}$ or $f=\sum_j\be_j b_j$ in $N_{n,d}$ for some $\be_j\in\FF^{\ast}$, $b_j\in\M$ satisfying $b_j>a_i$ for all $i,j$.

\item[2.] If $f=\sum_k f_k\in\M_{\FF}$ and $f_k\eqIII 0$ satisfies conditions from part~1 for all $k$, then $f\eqIII 0$. 
\end{enumerate}
Given $h\in\M_{\FF}$, we write $f\eqIII h$ if $f-h\eqIII 0$.
\bigskip

It is not difficult to see that $\eqIII$ is actually an equivalence on the vector space $\M_{\FF}$, i.e., $\eqIII$ have properties of transitivity and linearity over $\FF$. Note that part~2 of Definition~\ref{def_eqIII} is necessary for $\eqIII$ to be an equivalence.

\begin{theo}\label{theo_main}
Let $p=0$ or $p>\frac{n}{2}$. Then 
$$C_{n,d}\leq d\sum_{i=2}^n (i-1) C_{\left[n/i\right],d}+1.$$
\end{theo}
\begin{proof} There exists a $w\in\M$ with $\deg(w)=C_{n,d}-1$ and $w\neq 0$ in $N_{n,d}$. Moreover, by Remark~\ref{remark1} and $\NN$-homogeneity of $N_{n,d}$ we can assume that $w\not\eqIII 0$. Given a letter $x$, we write $d(x^i)$ for the number of $i^{\rm th}$ in the $x$-power of $w$, i.e.,%
$$\pwr_x(w)^{\rm ord} = (\al_1,\ldots,\al_r,\underbrace{i,\ldots,i}_{d(x^i)},\be_1,\ldots,\be_s),$$
where $\al_r<i<\be_1$. Obviously, $d(x^i)=0$ for $i\geq n$. 

Let $2\leq i\leq n$ and $x$ be a letter. Then $n=k i+r$ for $k=[n/i]\geq1$ and $0\leq r<i$. Consider elements $a_1,\ldots,a_k\in\M^{\neg x}$ and $\un{\theta}=((i-1)k+r,1^k)$. Note that for $a_{\si}=x^{i-1}a_{\si(1)}\cdots x^{i-1} a_{\si(k)}x^{i-1}$, $\si\in S_k$, the following statements hold:
\begin{enumerate} 
\item[$\bullet$] $a_{\si}\gtrless a_{\tau}$ for all $\si,\tau\in S_k$.

\item[$\bullet$] Let $i_1,\ldots,i_s>0$ satisfy $i_1+\cdots+i_s=(i-1)(k+1)$ and $e_0,\ldots,e_s\in\M_1$ be such products of $a_1,\ldots,a_k$ that for every $1\leq j\leq k$ $a_j$ is a factor of one and only element from the set $\{e_0,\ldots,e_s\}$. Moreover, we assume that $e_1,\ldots,e_{s-1}\in\M$. Define $e=e_0 x^{i_1}e_1 x^{i_2}\cdots x^{i_s}e_s\neq a_{\si}$ for all $\si\in S_k$. Then $e> a_{\si}$ for all $\si\in S_k$. 
\end{enumerate}%
\noindent{}To prove the second claim, we notice that there are two cases. Namely, in the first case $s=k+1$, $e_0=e_{k+1}=1$, and $e_1=a_{\tau(1)},\ldots,e_k=a_{\tau(k)}$ for some $\tau\in S_k$; and in the second case $\#\pwr_x(e)<\#\pwr_x(a_\si)$ for all $\si\in S_k$. In both cases we have $\pwr_x(e)^{\rm ord}>\pwr_x(a_\si)^{\rm ord}$ and $\pwr_y(e)^{\rm ord}\geq\pwr_y(a_\si)^{\rm ord}$ for any letter $y\neq x$ and any $\si\in S_k$. The claim is proven.    

Since $T_{\un{\theta}}(x,a_1,\ldots,a_k)x^{i-r-1}=0$ in $N_{n,d}$, we have 
$\sum_{\si\in S_k} a_{\si}\eqIII 0$. Moreover,
\begin{eq}\label{eq0}
\sum_{\si\in S_k} v\, a_{\si}w\eqIII 0
\end{eq}%
for all $v,w\in\M_1$ such that if $v\neq 1$ ($w\neq 1$, respectively), then its last (first, respectively) letter is not $x$.

Let $D=2^k-1$. Since $p=0$ or $p>\frac{n}{2}\geq k$, the Nagata--Higman Theorem implies that $C_{k,D}\leq 2^k-1$. For short, we write $C$ for $C_{k,D}$. Thus $y_1\cdots y_C=0$ in $N_{k,D}(y_1,\ldots,y_D)$, where $y_1,\ldots,y_D$ are new letters.  
Since $y_1\cdots y_C$ is multilinear, an equality \begin{eq}\label{eq1}
y_1\cdots y_C=\sum_{\un{u}} \al_{\un{u}}\, u_0 T_{1^k}(u_1,\ldots,u_k)u_{k+1}
\end{eq}%
holds in $\M_{\FF}(y_1,\ldots,y_C)$, where the sum ranges over $(k+2)$-tuples $\un{u}=(u_0,\ldots,u_{k+1})$ such that $u_0,u_{k+1}\in\M_1(y_1,\ldots,y_C)$, $u_1,\ldots,u_k\in\M(y_1,\ldots,y_C)$, and the number of non-zero coefficients $\al_{\un{u}}\in\FF$ is finite.  

Given $b_1,\ldots,b_C\in\M^{\neg x}$ and $0\leq l\leq k+1$, denote by $v_l\in\M_1$ the result of substitution $y_j\to x^{i-1}b_j$ ($1\leq j\leq C$) in $u_l$. We apply these substitutions to equality~\Ref{eq1} and multiply the result by $x^{i-1}$. Thus,
$$x^{i-1}b_1\cdots x^{i-1}b_C x^{i-1}=\sum_{\un{u}} \al_{\un{u}}\,v_0 T_{1^k}(v_1,\ldots,v_k)v_{k+1} x^{i-1}$$
in $\M_{\FF}=\M_{\FF}(x_1,\ldots,x_d)$. For every $\un{u}$ there exist $a_1,\ldots,a_k\in\M^{\neg x}$ satisfying $v_l=x^{i-1}a_l$ for all $1\leq l\leq k$. If $u_{k+1}\neq1$, then we also have $v_{k+1}=x^{i-1}a_{k+1}$ for some $a_{k+1}\in\M^{\neg x}$. Since $T_{1^k}(v_1,\ldots,v_k)=\sum_{\si\in S_k} v_{\si(1)}\cdots v_{\si(k)}$, we have 
$$T_{1^k}(v_1,\ldots,v_k)v_{k+1}x^{i-1}=\sum_{\si\in S_k} a_{\si} f,$$%
where $f$ stands for $1$ in case $u_{k+1}=1$ and for $a_{k+1}x^{i-1}$ in case $u_{k+1}\neq1$.  Combining the previous two equalities with equivalence~\Ref{eq0}, we obtain 
\begin{eq}\label{eq2}
x^{i-1}b_1\cdots x^{i-1}b_C x^{i-1}\eqIII 0.
\end{eq}%
Hence, the equivalence $b_0 x^{i-1}b_1\cdots x^{i-1} b_{C+1}\eqIII 0$ holds for all $b_1,\ldots,b_{C}\in\M^{\neg x}$ and $b_0,b_{C+1}\in\M_1$ such that if $b_0\neq 1$ ($b_{C+1}\neq 1$, respectively), then its last (first, respectively) letter is not $x$. Since $w\not\eqIII 0$, we obtain 
$$d(x^{i-1})\leq C_{[n/i],d},$$ 
and therefore $\deg_x(w)\leq \sum_{1<i\leq n} (i-1)\, C_{[n/i],d}$ for every letter $x$. The proof is completed. 
\end{proof}

\begin{remark}\label{remark_main} Since $C_{1,d}=1$, we can reformulate the statement of Theorem~\ref{theo_main} as follows. Let $p=0$ or $p>\frac{n}{2}$ and $m=[n/2]$. Then  $C_{n,d}\leq A_n d + 1$, where 
$$A_n=\sum_{i=2}^m (i-1)C_{[n/i],d} + \frac{1}{2}(n+m-1)(n-m).$$
\end{remark}

\example\label{ex_n5}
To illustrate the proof of Theorem~\ref{theo_main}, we repeat this proof in the partial case of $n=5$ and $p\neq 2$. We write $a,b,c$ for some elements from $\M^{\neg x}$.

Let $i=2$. Then $k=[n/i]=2$ and $r=1$. Since $T_{311}(x,a,b)=0$ in $N_{5,d}$, we have the following partial case of~\Ref{eq0}:
\begin{eq}\label{eq_n5}
xaxbx + xbxax \eqIII 0.
\end{eq}%
Note that $C_{2,D}=3$ for all $D\geq2$. We rewrite the proof of this fact, using formula~\Ref{eq_n5} instead of the equality $uv+vu=0$ in $N_{2,D}$:
$$x a x \cdot b x c \cdot x \eqIII %
- x b ( x c x a x )\eqIII %
  (x b x a x ) c x \eqIII %
- x a x b x c x.$$
Here we use dots and parentheses to show how we apply~\Ref{eq_n5}. Thus we obtain the partial case of formula~\Ref{eq2}: $x a x b x c x \eqIII 0$. 
Therefore, $d(x)\leq 3$.

Let $i=3$. Then $k=[n/i]=1$ and $r=2$. Since $T_{41}(x,a)=0$ in $N_{5,d}$, we have $x^2 a x^2 \eqIII 0$. Considering $i=4,5$, we can see that $x^3 a x^3 \eqIII 0$ and $x^4 a x^4 \eqIII 0$. Thus, $d(x^j)\leq C_{1,D}=1$ for $j=2,3,4$.

The obtained restrictions on $d(x^j)$ for $1\leq j\leq 4$ imply that $\deg{w}\leq 12d$. Hence, $C_{5,d}\leq 12d+1$.

\section{Polynomial bound}\label{section_poly}

This section is dedicated to the proof of the next result.

\begin{cor}\label{cor_poly}
If $p>\frac{n}{2}$, then $C_{n,d}<n^{\log_2(3d+2)+1}$.
\end{cor}
\bigskip

Theorem~\ref{theo_main} together with the inequality $C_{j-1,d}\leq C_{j,d}$ for all $j\geq 2$ implies that 
$$C_{n,d}\leq d\sum_{j=1}^k \ga_j\, C_{\left[n/2^j\right],d} + 1$$
for $\ga_j=(2^j-1)+2^j+\cdots+(2^{j+1}-2)=3(2^{j}-1)2^{j-1}$ and 
$k>0$ satisfying $1\leq \frac{n}{2^k}<2$. Thus,
\begin{eq}\label{eq_key}
C_{n,d}<\frac{3d}{2}\sum_{j=1}^k 4^j\, C_{\left[n/2^j\right],d},
\end{eq}%
where $\frac{n}{2}<2^k\leq n$. 

Let us fix some notations. If $a$ is an arrow in an oriented graph, then we denote the head of $a$ by $a'$ and the tail of $a$ by $a''$, i.e.,
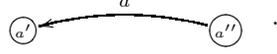
$$\vcenter{
\xymatrix@C=1cm@R=1cm{ %
\vtx{a'}\ar@/^/@{<-}[rr]^{a} && \vtx{a''}\\
}} \quad.
$$
We say that $a''$ is a {\it predecessor} of $a'$ and $a'$ is a {\it successor} of $a''$. 

For every $l\geq1$ we construct an oriented tree $T_l$ as follows. 
\begin{enumerate}
\item[$\bullet$] The underlying graph of $T_l$ is a tree.

\item[$\bullet$] Vertices of $T_l$ are marked with $0,\ldots,l$.

\item[$\bullet$] Let a vertex $v$ be marked with $i$. Then $v$ has exactly $i$ successors, marked with $0,1,\ldots,i-1$. If $i<l$, then $v$ has exactly one predecessor. If $i=l$, then $v$ does not have a predecessor and it is called the {\it root} of $T_l$. 

\item[$\bullet$] If $a$ is an arrow of $T_l$ and $a',a''$ are marked with $i,j$, respectively, then $a$ is marked with $4^{j-i}\de$, where $\de=3d/2$.
\end{enumerate}

\example\label{ex1} 
$$T_1:\qquad%
\xymatrix@C=1.3cm@R=1.3cm{ %
\vtx{1} \ar@/^/@{->}[d]_{4\de} \\ %
\vtx{0}\\
}
\qquad\qquad\qquad
T_3:\qquad%
\xymatrix@C=1.3cm@R=1.3cm{ %
\vtx{3} \ar@/^/@{->}[d]^{4\de} \ar@/^/@{->}[drr] \ar@/^/@{->}[drrr] &&& \\%
\vtx{2} \ar@/^/@{->}[d]_{4\de} \ar@/^/@{->}[dr]&& 
\vtx{1} \ar@/^/@{->}[d]_{4\de} & \vtx{0} \\
\vtx{1} \ar@/^/@{->}[d]_{4\de} & \vtx{0} & \vtx{0} &   \\
\vtx{0}&&&\\
}
\begin{picture}(0,0) 
\put(-105,-23){$\scriptstyle 4^2\de$}%
\put(-48,-23){$\scriptstyle 4^3\de$}%
\put(-130,-69){$\scriptstyle 4^2\de$}%
\end{picture}$$%
Here we write a number that is prescribed to a vertex (an arrow, respectively) in this vertex (near this arrow, respectively). 
\bigskip

If $b$ is an oriented path in $T_l$, then we write $\deg{b}$ for the number of arrows in $b$ and $|b|$ for the product of numbers assigned to arrows of $b$. Denote by $P_l$ the set of maximal (by degree) paths in $T_l$. Note that there is 1-to-1 correspondence between $P_l$ and the set of leaves of $T_l$, i.e., vertices marked with $0$. 
We claim that
$$C_{n,d}<\sum_{b\in P_k} |b|.$$
To prove this statement we use induction on $n\geq2$. If $n=2$, then $k=1$ and $C_{2,d}<4\de$ by~\Ref{eq_key}, and therefore the statement holds. For $n>2$ formulas~\Ref{eq_key} and $[[n/2^{j_1}]/2^{j_2}]=[n/2^{j_1+j_2}]$ for all $j_1,j_2>0$ together with the induction hypothesis imply that 
$$C_{n,d}<\sum_{j=1}^k \sum_{b\in P_{k-j}} 4^j\de\ |b|.
$$
The statement is proven.

Since the sum of exponents of $4$ along every maximal path is $k$, we obtain that
\begin{eq}\label{eq_formula}
C_{n,d}<\sum_{b\in P_k} 4^{k} \left(\frac{3d}{2}\right)^{\deg{b}}.
\end{eq}%

Given $1\leq r\leq k$, denote by $P_{k,r}$ the set of $b\in P_{k}$ with $\deg{b}=r$. We claim that 
\begin{eq}\label{eq_claim}
\# P_{k,r}=\binom{k-1}{r-1}, 
\end{eq}%
where $\# P_{k,r}$ stands for the cardinality of $P_{k,r}$. To prove the claim we notice that $P_{k,r}$ is the set of $r$-tuples $(j_1,\ldots,j_r)$ satisfying $j_1,\ldots,j_r\geq1$ and $j_1+\cdots+j_r=k$. Hence $\# P_{k,r}$ is equal to the cardinality of the set of all $(r-1)$-tuples $(q_1,\ldots,q_{r-1})$ such that $1\leq q_1<\cdots< q_{r-1}\leq k-1$ since we can set $j_1=q_1$, $j_2=q_2-q_1,\ldots, j_r=k-q_{r-1}$. The claim is proven.

Applying~\Ref{eq_claim} to inequality~\Ref{eq_formula}, we obtain
$$C_{n,d}< 4^k \sum_{r=1}^k \left(\frac{3d}{2}\right)^r \binom{k-1}{r-1} = %
4^k \frac{3d}{2}\sum_{r=0}^{k-1} \left(\frac{3d}{2}\right)^r \binom{k-1}{r}= %
4^k \frac{3d}{2} \left( 1+ \frac{3d}{2}\right)^{k-1}.$$
Thus,
$$C_{n,d}<4^k \left( 1+ \frac{3d}{2}\right)^{k}.$$
Since $2^k\leq n$, we have 
$$C_{n,d}<n^2 \left( 1+ \frac{3d}{2}\right)^{\log_2(n)}=n^{\log_2\left(1+\frac{3d}{2}\right)+2} = 
n^{\log_2(3d+2)+1}.$$
Corollary~\ref{cor_poly} is proven.

\section{Corollaries}\label{section_cor}

\begin{cor}\label{cor1}
Let $p>\frac{n}{2}$. Then $C_{n,d}<4\cdot 2^{n/2}d$. Moreover, if $n\geq30$, then $C_{n,d}<2\cdot 2^{n/2}d$. 
\end{cor}
\bigskip

We split the proof of Corollary~\ref{cor1} into several lemmas. Let  $m=[n/2]$.  
For $2\leq i\leq m$ denote 
$\ga_i=(i-1) 2^{n/i}$ and $\de_n=2^{n/2}+2^{n/3}(n-4)+\frac{1}{4}(n+1)^2$. 

\begin{lemma}\label{lemma_tec1}
For $3\leq i\leq m$ the inequality $\ga_i\leq \ga_3$ holds.  
\end{lemma} 
\begin{proof} The required inequality is equivalent to the following one:
\begin{eq}\label{eq_ineq1}
i-1\leq 2\cdot 2^{n\frac{i-3}{3i}}.
\end{eq}%
\indent{}Let $i=4$. Then $n\geq 8$ and it is not difficult to see that the inequality $3\leq 2\cdot 2^{n/12}$ holds. 

Let $i\geq 5$. Then inequality~\Ref{eq_ineq1} follows from $i-1\leq 2\cdot 2^{2n/15}$. Since $i-1\leq \frac{n}{2}$, the last inequality follows from $n\leq 4\cdot 2^{2n/15}$, which holds for all $n\geq2$. 
\end{proof}

\begin{lemma}\label{lemma_tec2}
For $n\geq2$ the inequality $\de_n\leq 4\cdot 2^{n/2}-1$ holds. Moreover, $\de_n\leq 2 \cdot 2^{n/2}-1$ in case $n\geq 30$. 
\end{lemma} 
\begin{proof}
Let $n\geq 30$. Then it is not difficult to see that $2\cdot 2^{n/2}-1-\de_n=\left(2^{n/2}-n\cdot 2^{n/3}\right)+\left(4\cdot 2^{n/3}-\frac{1}{4}(n+1)^2-1\right)\geq 0$. If $2\leq n <30$, then performing calculations we can see that the claim of the lemma holds. 
\end{proof}

Now we can prove Corollary~\ref{cor1}:

\begin{proof} If $n=2$ or $n=3$, respectively, then $C_{n,d}\leq\max\{3,d\}$ or $C_{n,d}\leq 3d+1$, respectively (see Section~\ref{section_intro}), and the required is proven.   

Assume that $n\geq4$. By Remark~\ref{remark_main}, $C_{n,d}\leq A_n d+1$. Since $p>[n/i]$ for $2\leq i\leq m$, the Nagata--Higman Theorem implies $C_{[n/i],d}\leq 2^{n/i}-1$. Thus,
$$A_n\leq \sum_{2\leq i\leq m} \ga_i + \be_n,$$
where $\be_n=\frac{1}{2}\left(-m(m-1) + \frac{}{}\!(m+n-1)(n-m)\right)$. Separately considering the cases of $n$ even and odd, we obtain that $\be_n\leq (n+1)^2/4$. Since $m\geq2$, Lemma~\ref{lemma_tec1} implies that 
$$\sum_{2\leq i\leq m} \ga_i\leq \ga_2 + \ga_3(m-2).$$  
It follows from the above mentioned upper bound on $\be_n$ and the inequality $m\leq\frac{n}{2}$ that $A_n\leq \de_n$. Lemma~\ref{lemma_tec2} completes the proof.
\end{proof}

To prove Corollary~\ref{cor1a} (see below) we need the following slight improvement of the upper bound from Nagata--Higman Theorem. 

\begin{lemma}\label{lemma_NH2}
If $p>n$, then $C_{n,d}< 7 \cdot 2^{n-3}$ for all $n\geq3$.
\end{lemma}
\begin{proof}
If $n=3$, then the claim of the lemma follows from $C_{3,d}=6$ (see Section~\ref{section_intro}). 

It is well known that 
\begin{eq}\label{eq_wellknown}
n x^{n-1}a y^{n-1}=0 
\end{eq}%
in $N_{n,d}$ for all $a,x,y$ (see~\cite{Jacobson64}). Thus, $C_{n,d}\leq 2C_{n-1,d}+1$. Applying this formula recursively,  we obtain that $C_{n,d}\leq 2^{n-3} C_{3,d} + \sum_{i=0}^{n-4} 2^i$ for $n\geq4$. Since $p>4$, the equality $C_{3,d}=6$ concludes the proof.

\end{proof}

\begin{cor}\label{cor1a}
Let $4\leq n\leq 9$ and $\frac{n}{2}<p\leq n$. Then $C_{n,d}\leq a_n d+1$, where $a_4=8$, $a_5=12$, $a_6=24$, $a_7=30$, $a_8=50$, $a_9=64$.
\end{cor}
\begin{proof}
We have $C_{2,d}=3$ in case $p>2$ and $C_{3,d}=6$ in case $p>3$ (see Section~\ref{section_intro}) By Lemma~\ref{lemma_NH2}, $C_{4,d}\leq 13$ in case $p>4$. Applying the upper bound on $C_{n,d}$ from Theorem~\ref{theo_main} recursively and using the above given estimations on $C_{k,d}$ for $k=2,3,4$, we obtain the required.
\end{proof}

The following conjecture is a generalization of Razmyslov's upper bound to the case of $p>n$ and it holds for $n=2,3$:

\begin{conj}\label{conj_n2}
For all $n,d\geq 2$ and $p>n$ we have $C_{n,d}\leq n^2$.
\end{conj}

\begin{cor}\label{cor_modulo_conj}
Assume that Conjecture~\ref{conj_n2} holds. Then $C_{n,d}<n^2 \ln(n)\, d$ for $\frac{n}{2}<p\leq n$.  
\end{cor}
\begin{proof}
For $n=2,3$ the claim holds by Section~\ref{section_intro}.   

Assume that $n\geq4$. By Remark~\ref{remark_main}, $C_{n,d}\leq A_n d+1$. Since $p>[n/i]$, Conjecture~\ref{conj_n2} implies 
$$A_n\leq \sum_{2\leq i\leq m} (i-1)\frac{n^2}{i^2} + \be'_n,$$
where $\be'_n=\frac{1}{2}(m+n-1)(n-m)$. Separately considering the cases of $n$ even and odd, we obtain that $\be'_n\leq 3n^2/8$. Denote by  $\xi_m$ the $m^{\rm th}$ harmonic number $1+\frac{1}{2}+\frac{1}{3}+\cdots+\frac{1}{m}$. We have
$$A_n< n^2 (\xi_m-1)+\frac{3}{8}n^2 -1.$$
Since $\xi_m<\ln{m}+\ga + \frac{1}{2m}$, where $\ga<1$ is Euler's constant (for example, see pages 73 and 79 of~\cite{Havil_book}), 
$$A_n<n^2\left(\ln{m} + \frac{5}{8}\right)-1 < n^2\ln(n)-1$$
and we obtain the required inequality.
\end{proof}

\begin{remark}\label{remark_Belov}{} Using another approach, in recent paper~\cite{Belov2011} Belov and Kharitonov obtained the following upper bounds on $C_{n,d}$ for all $p$:
\begin{enumerate}
\item[1)] $C_{n,d}\leq 4^{\log_3(64)+5}\cdot (n^{12})^{\log_3(4n)+1}d$ (Corollary~1.16 from~\cite{Belov2011});

\item[2)] $C_{n,d}\leq 256\cdot n^{8\log_2(n)+22}d$ (see Theorem~1.17 from~\cite{Belov2011});
\end{enumerate}
where the second estimation is better for small $n$. These bounds are linear with respect to $d$ and subexponential with respect to $n$. 

Let us compare bounds~1) and~2) with the bound from Corollary~\ref{cor1} in case $p>\frac{n}{2}$: $C_{n,d}<4\cdot 2^{n/2}d$. If $n>\!\!>\!0$ is large enough, then bounds~1) and~2) are essentially better than the bound from Corollary~\ref{cor1}. On the other hand, 
for $4\leq n\leq 2000$ the bound from Corollary~\ref{cor1} is at least $10^{20}$ times better than bounds~1) and~2). This claim follows from straightforward computations.
\end{remark}

\section{The case of $n=4$}\label{section_n_is_4}

\begin{theo}\label{theo_n_is_4} 
For $d\geq 2$ we have 
\begin{enumerate}
\item[$\bullet$] $C_{4,d}=10$, if $p=0$;

\item[$\bullet$] $3d< C_{4,d}$, if $p=2$;

\item[$\bullet$] $3d+1\leq C_{4,d}\leq 3d+4$, if $p=3$;

\item[$\bullet$] $10\leq C_{4,d}\leq 13$, if $p>3$.
\end{enumerate}
\end{theo}
\bigskip

In what follows we assume that $n=4$ and $p\neq 2$ unless otherwise stated. To prove Theorem~\ref{theo_n_is_4} (see the end of the section), we introduce a new $\approx$-equivalence on $\M_{\FF}$ as follows.
Given $\un{\al}\in\NN^r$ and $\un{\be}\in\NN^s$ ($r,s\geq0$), we write 
$$\un{\al}\succ\un{\be}\;\text{ if }\;r<s.$$%
Using $\succ$ instead of $>$,  we introduce the partial order $\succ$ on $\M$ similarly to Definition~\ref{def_order}. Then, using the partial order $\succ$ on $\M$ instead of $>$, we introduce the $\approx$-equivalence on $\M_{\FF}$ similarly to the $\eqIII$-equivalence (see Definition~\ref{def_eqIII}). The resulting definition of $\approx$ is the following one:

\definition{of the $\approx$-equivalence on $\M_{\FF}$}\label{def_approx} 
\begin{enumerate}
\item[1.] Let $f=\sum_i\al_i a_i\in\M_{\FF}$, where $\al_i\in\FF^{\ast}$, $a_i\in\M$, and  $\#\pwr_y(a_i)=\#\pwr_y(a_{i'})$ for every letter $y$ and all $i,i'$.
Then $f\approx 0$ if $f=0$ in $N_{n,d}$ or $f=\sum_j\be_j b_j$ in $N_{n,d}$ for $\be_j\in\FF^{\ast}$, $b_j\in\M$ satisfying 
\begin{enumerate}
\item[$\bullet$] $\#\pwr_x(a_i)>\#\pwr_x(b_j)$ for some letter $x$,

\item[$\bullet$] $\#\pwr_y(a_i)\geq\#\pwr_y(b_j)$ for every letter $y$
\end{enumerate}
for all $i,j$;

\item[2.] If $f=\sum_k f_k\in\M_{\FF}$ and $f_k\approx 0$ satisfies conditions from part~1 for all $k$, then $f\approx 0$. 
\end{enumerate}
Given $h\in\M_{\FF}$, we write $f\approx h$ if $f-h\approx 0$.
\bigskip

\begin{remark} Note that the partial order $>$ on $\M$ is stronger than $\succ$. Namely, for $a,b\in\M$ we have
\begin{enumerate}
\item[$\bullet$] if $a \succ b$, then $a>b$;

\item[$\bullet$] if $a > b$, then $a \succ b$ or $a\approx b$.
\end{enumerate}
Therefore, $\eqIII$-equivalence on $\M_{\FF}$ is weaker than $\approx$-equivalence. Namely, for $f,h\in\M_{\FF}$ the equality $f\approx h$ implies $f\eqIII h$, but the converse statement does not hold.
\end{remark}
\bigskip

Let $a,b,c,a_1,\ldots,a_4$ be elements of $\M$. By definition, 
\begin{enumerate}
\item[$\bullet$] $T_{4}(a)=a^4$,

\item[$\bullet$] $T_{31}(a,b)=a^3 b + a^2 b a + a b a^2 + b a^3$,

\item[$\bullet$] $T_{211}(a,b,c)=a^2bc + a^2cb + ba^2c + ca^2b + bca^2 + cba^2 +
abca + acba + abac + acab + baca + cabc$,

\item[$\bullet$] $T_{22}(a,b)=a^2 b^2 + b^2 a^2 + abab + baba + a b^2 a + b a^2 b$,

\item[$\bullet$] $T_{1^4}(a_1,\ldots,a_4)=\sum_{\si\in S_4} a_{\si(1)}\cdots a_{\si(4)}$
\end{enumerate}
(see Section~\ref{section_main_theo}). Then 
$$T_4(a)=0,\quad T_{31}(a,b)=0,\quad T_{211}(a,b,c)=0, \quad T_{22}(a,b)=0, \quad 
T_{1^4}(a_1,\ldots,a_4)=0$$  
are relations for $N_{4,d}$, which generate the ideal of relations for $N_{4,d}$. Multiplying $T_{31}(a,b)$ by $a$ several times we obtain that equalities
\begin{eq}\label{eq_rel2}
a^3 b a + a^2 b a^2 + a b a^3 = 0, 
\end{eq} \vspace{-5mm}
\begin{eq}\label{eq_rel3}
a^3 b a^2 + a^2 b a^3 = 0, 
\end{eq} \vspace{-5mm}
\begin{eq}\label{eq_rel4}
a^3 b a^3 = 0 
\end{eq}%
hold in $N_{4,d}$.

\begin{remark}\label{remark_back}
Let $f\in\M_{\FF}$. Denote by ${\rm inv}(f)$ the element of $\M_{\FF}$ that we obtain by  reading $f$ from right to left. As an example, for $f=x_1^2x_2-x_3$ we have ${\rm inv}(f)=-x_3+x_2x_1^2$. 

Obviously, if $f=0$ in $N_{n,d}$, then ${\rm inv}(f)=0$ in $N_{n,d}$. Similar result also holds for $\approx$-equivalence.
\end{remark}

\begin{lemma} Let $x$ be a letter and $a,b,c\in\M^{\neg x}$. Then the next relations are valid in $N_{4,d}$:
\begin{eq}\label{eq_rel5}
x^3 a x b x^2 = - x^3 a x^2 b x,\quad x a x^3 b x^2 = x^3 a x^2 b x. 
\end{eq}%
Moreover, the following equivalences hold:
\begin{eq}\label{eq_rel_new}
xax^2\approx -x^2ax,   
\end{eq}\vspace{-5mm}
\begin{eq}\label{eq_rel6}
x^i a x b x \approx 0,\quad x a x^i b x \approx 0,\quad  x a x b x^i \approx 0
\end{eq}%
for $i=2,3$,
\begin{eq}\label{eq_rel_L1}
x a x b x c x\approx 0.   
\end{eq}%
\end{lemma}
\begin{proof} We have 
$$x^3aT_{31}(x,b)= x^3 a x^3 b + x^3 a x^2 b x + x^3 a x b x^2 + x^3 ab x^3 =0$$
in $N_{4,d}$. By equality~\Ref{eq_rel4}, $x^3 a x b x^2 =  - x^3 a x^2 b x$ in $N_{4,d}$. Similarly we can see that
$$T_{31}(x, a x^3 b) = x^3 a x^3 b + x^2 a x^3 b x + x a x^3 b x^2 + a x^3 b x^3  = 
 x^2 a x^3 b x + x a x^3 b x^2  = 0$$
in $N_{4,d}$. By~\Ref{eq_rel3}, $x^2 a x^3 b x = -x^3 a x^2 b x$ in $N_{4,d}$ and equalities~\Ref{eq_rel5} are proven.

Since $T_{31}(x,a)=0$ in $N_{4,d}$, equivalence~\Ref{eq_rel_new} is proven.

Let $i=2$. By~\Ref{eq_rel_new}, $x a x b x^2  \approx - x a x^2 b x \approx x^2 a x b x$. 
On the other hand, \Ref{eq_rel_new} implies $x a x b x^2  \approx -x^2 a x b x$. Equivalences~\Ref{eq_rel6} for $i=2$ are proven.

Let $i=3$. Since $T_{211}(x,a,x^3b)=0$ and $x^3 T_{211}(x,a,b)=0$ in $N_{4,d}$, we have
$$xax^3bx + x^3 bxax\approx0 \quad \text{and}\quad 
x^3axbx + x^3bxax\approx 0,$$
respectively. Thus, $x^3axbx \approx xax^3bx$. Using Remark~\ref{remark_back}, we obtain 
\begin{eq}\label{eq_star1}
x^3axbx \approx xax^3bx\approx xaxbx^3.
\end{eq}%
The equality $x^2aT_{31}(x,a)=0$ implies 
$$x^2 axb x^2 + x^2 a x^2 bx \approx0.$$
Applying relation~\Ref{eq_rel2}, we obtain 
$$x^3axbx + xaxbx^3 + x^3axbx + xax^3bx\approx 0.$$ 
Equivalences~\Ref{eq_star1} complete the proof of~\Ref{eq_rel6}.

Since $T_{211}(x,a,bxc)x=0$ and $T_{211}(x,a,b)xcx=0$ in $N_{4,d}$, we obtain $$xaxbxcx  + xbxcxax \approx 0\quad \text{and}\quad xaxbxcx + xbxaxcx \approx 0,$$ respectively. The equality $x b T_{211}(x,a,c)x=0$ in $N_{4,d}$ implies 
$$xbxcxax + xbxaxcx   \approx0,$$ 
and therefore $xaxbxcx\approx0$. 
\end{proof}

If $\un{\al}\in\NN^r$, $\un{\be}\in\NN^s$, then we write $\un{\al}\subset \un{\be}$ and say that $\un{\al}$ is a subvector of $\un{\be}$ if there are $1\leq i_1<\cdots <i_r$ such that $\al_{1}=\be_{i_1},\ldots,\al_{r}=\be_{i_r}$. 

\begin{lemma}\label{lemma_canonical_form}
If $f\in\M_{\FF}$, then $f=0$ in $N_{4,d}$ or $f=\sum_i\al_i a_i$ in $N_{4,d}$ for some  $\al_i\in\FF^{\ast}$, $a_i\in\M$ such that for every letter $x$ $\pwr_x(a_i)$ belongs to the following list:
\begin{enumerate}
\item[$\bullet$] $\emptyset$, $(1)$, $(1,1)$, $(1,1,1)$,

\item[$\bullet$] $(2)$, $(2,1)$,

\item[$\bullet$] $(3)$, $(3,1)$, $(1,3)$, $(3,2)$, $(3,2,1)$.
\end{enumerate} 
Moreover, we can assume that for all pairwise different letters $x,y,z$ and all $i$ the following conditions do not hold:   
\begin{enumerate}
\item[a)] $\pwr_x(a_i)=(3,2,1)$ and $(3)\subset \pwr_y(a_i)$; 

\item[b)] $(3)$ is a subvector of $\pwr_x(a_i)$, $\pwr_y(a_i)$, and $\pwr_z(a_i)$;

\item[c)] $(3,2)$ is a subvector of $\pwr_x(a_i)$ and $\pwr_y(a_i)$. 
\end{enumerate}
\end{lemma}
\begin{proof}  Let $x$ be a letter and $f=\sum_{j\in J} \be_j b_j$ for $\be_j\in\FF^{\ast}$ and $b_j\in\M$. We claim that the statement of the lemma holds for $f$ for the given letter $x$.  To prove the claim we use induction on $k=\max\{\#\pwr_x(b_j)\,|\,j\in J\}$.

If $k=0,1$, then the claim holds. 

If $b_j=b_{1j}x^2 b_{2j} x^2 b_{3j}$ for some $b_{1j},b_{2j},b_{3j}\in\M^{\neg x}$, then $a_j=-b_{1j}x^3 b_{2j} x b_{3j} - b_{1j}x b_{2j} x^3 b_{3j}$ in $N_{4,d}$ by relation~\Ref{eq_rel2}. Note that $\#\pwr_x(b_j) = \#\pwr_x(b_{1j}x^3 b_{2j} x b_{3j}) =\#\pwr_x(b_{1j}x b_{2j} x^3 b_{3j})$. Moreover, if $(2,\ldots,2)\subset\pwr_x(b_j)$, then we apply~\Ref{eq_rel2} several times. Therefore, without loss of generality can assume that $(2,2)$ is not a subvector of $\pwr_x(b_j)$ for all $j$.

If one of the vectors
$$(r), r>3;\; (3,3);\;(s,1,1), \;(1,s,1), \;(1,1,s),  s\in\{2,3\};\; (1,1,1,1)$$
is a subvector of $\pwr_x(b_j)$, then $b_j\approx0$ by the equality $x^4=0$ in $N_{4,d}$ and formulas~\Ref{eq_rel4},~\Ref{eq_rel6},~\Ref{eq_rel_L1}, respectively. Thus, $f\approx 0$ or $f\approx \sum_{j\in J_0}\be_j b_j$ for such $J_0\subset J$ that for every $j\in J_0$ the vector $\pwr_x(b_j)$ up to permutation of its entries belongs to the following list:
$$\emptyset,\, (1),\, (1,1),\, (1,1,1),\, (2),\, (2,1),\,  
(3),\, (3,1),\,  (3,2),\, (3,2,1).$$%
Let $j\in J_0$. If $\pwr_x(b_j)=(\si(1),\si(2),\si(3))$ for some $\si\in S_3$, then applying relations~\Ref{eq_rel3} and~\Ref{eq_rel5} we obtain that $b_j=\pm c_j$ in $N_{4,d}$ for a monomial $c_j\in\M$ satisfying $\pwr_x(c_j)=(3,2,1)$. If $\pwr_x(b_j)$ is $(1,2)$ or $(2,3)$, then we apply formulas~\Ref{eq_rel_new} or~\Ref{eq_rel3}, respectively, to obtain that $b_j \approx -c_j$ for a monomial $c_j\in\M$ with $\pwr_x(c_j)\in\{(2,1),\,(3,2)\}$.  So we get that $f\approx h$ for such $h\in\M_{\FF}$ that the claim holds for $h$. The induction hypothesis and Definition~\ref{def_approx} complete the proof of the claim.

Let $y$ be a letter different from $x$. Relations from the proof of the claim do not affect $y$-powers.  Therefore, applying the claim to $f$ for all letters subsequently, we complete the proof of the first part of the lemma. 

Consider an $a\in\M$. If $a$ satisfies condition~a), then relations~\Ref{eq_rel3} and~\Ref{eq_rel5} together with relation~\Ref{eq_wellknown} imply that $a=0$ in $N_{4,d}$. If $a$ satisfies condition~b) or~c), then relations~\Ref{eq_wellknown} and~\Ref{eq_rel3} imply that $a=0$ in $N_{4,d}$. Thus, the second part of the lemma is proven.
\end{proof}

The following lemma resembles Lemma~3.3 from~\cite{Lopatin_O3}.

\begin{lemma}\label{lemma_pi}
Let $p=2$ and $1\leq k\leq d$. For every homogeneous $f\in\M_{\FF}$ of multidegree $(\theta_1,\ldots,\theta_d)$ with $\theta_k\leq 3$ and $\theta_1+\cdots+\theta_{k-1}+\theta_{k+1}+\cdots+\theta_d>0$  we define $\pi_k(f)\in\M_{\FF}$ as the result of the substitution $x_k\to 1$ in $a$, where $1$ stands for the unity of $\M_1$. 

Then $f=0$ in $N_{4,d}$ implies $\pi_k(f)=0$ in $N_{4,d}$.   
\end{lemma}
\begin{proof} Let $a,b,c,u\in\M$. By definition, $\pi_k(ab)=\pi_k(a)\pi_k(b)$. 
Then by straightforward calculations we can show that 
$\pi_k(T_{31}(a,b))=0$, $\pi_k(T_{211}(a,b,c))=0$, $\pi_k(T_{22}(a,b))=0$, and $\pi_k(T_{1^4}(a,b,c,u))=0$ in $N_{4,d}$. The proof is completed.  
\end{proof}

We now can prove Theorem~\ref{theo_n_is_4}:
\begin{proof}
If $p=0$, then the required was proven by Vaughan--Lee in~\cite{Vaughan93}. If $p>3$, then the claim follows from Kuzmin's low bound (see Section~\ref{section_intro}) and Lemma~\ref{lemma_NH2}. 

Let $p=2$ and $a=x_1^3\cdots x_{d}^3$. Assume that $a=0$ in $N_{4,d}$. Applying $\pi_1,\ldots,\pi_{d-1}$ from Lemma~\ref{lemma_pi} to $a$ we obtain that $x_d^3=0$ in $N_{4,d}$; a contradiction. Thus, $C_{4,d}>\deg{a}=3d$. 

Assume that $p=3$. Consider an $a\in\M$ such that $a\neq0$ in $N_{4,d}$. Applying Lemma~\ref{lemma_canonical_form} to $a$, without loss of generality we can assume that $a$  satisfies all conditions from Lemma~\ref{lemma_canonical_form}. Denote $t_i=\deg_{x_i}(a)$ and $r=\#\{i\,|\,(3)\text{ is  subvector of }\pwr_{x_i}(a)\}$. Then
\begin{enumerate}
\item[a)] $t_i\leq 6$;

\item[b)] if $t_i\geq 4$, then $(3)\subset\pwr_{x_i}(a)$
\end{enumerate}
for all $1\leq i\leq d$.

If $r=0$, then $\deg(a)\leq 3d$ by part~b). If $r=1$, then $\deg(a)\leq 6 + 3(d-1)=3d+3$ by parts~a) and~b). 

Let $r=2$. Then without loss of generality we can assume that $(3)$ is a subvector of $\pwr_{x_1}(a)$ and $\pwr_{x_2}(a)$. Since condition~a) of Lemma~\ref{lemma_canonical_form} does not hold for $a$, $(3,2,1)$ is not a subvector of $\pwr_{x_i}(a)$ for $i=1,2$. Hence, $t_1,t_2<6$. If $t_1=t_2=5$, then condition~c) of Lemma~\ref{lemma_canonical_form} holds for $a$; a contradiction. Therefore, $t_1+t_2\leq 9$. By part~b), $t_i\leq 3$ for $3\leq i\leq d$. Finally, we obtain that $\deg(a)\leq 3d+3$.

If $r\geq3$, then $a$ satisfies condition~b) of Lemma~\ref{lemma_canonical_form}; a contradiction. 

So, we have shown that $\deg(a)\leq 3d+3$, and therefore $C_{4,d}\leq 3d+4$. On the other hand, $C_{4,d}\geq C_{3,d}=3d+1$ by~\cite{Lopatin_Comm2}. The proof is completed.
\end{proof}

\begin{remark}\label{remark_n_is_4} Assume that $n=4$ and $p=3$. Let us compare the upper bound $C_{4,d}\leq 3d+3$ from Theorem~\ref{theo_n_is_4} with the known upper bounds on $C_{4,d}$:
\begin{enumerate}
\item[$\bullet$] Corollary~\ref{cor1a} implies that $C_{4,d}<8d+1$;

\item[$\bullet$] bounds by Belov and Kharitonov~\cite{Belov2011} imply that $C_{4,d}\leq B_{4}d$, where $B_4>10^{20}$ (see Remark~\ref{remark_Belov} for details);

\item[$\bullet$] bounds by Klein~\cite{Klein00} imply that $C_{4,d}<\frac{2^{11}}{3}d^4$ and $C_{4,d}<2^{128}d^2$ (see Section~\ref{section_intro} for details). 
\end{enumerate}
\end{remark}

\section{$GL(n)$-invariants of matrices}\label{section_matrix_inv}
 
The general linear group $GL(n)$ acts on $d$-tuples 
$V=(\mathbb{F}^{n\times n})^{\oplus d}$
of $n\times n$ matrices over $\mathbb{F}$ by the
diagonal conjugation, i.e.,
\begin{equation}\label{eq_diag_conj}
g\cdot (A_1,\ldots,A_d)=(g A_1 g^{-1},\ldots,g A_d g^{-1}),
\end{equation}%
where $g\in GL(n)$ and $A_1,\ldots,A_d$ lie in $\mathbb{F}^{n\times n}$. The coordinate algebra of the affine variety $V$ is the algebra of polynomials $R=\mathbb{F}[V]=\mathbb{F}[x_{ij}(k)\,|\,1\leq i,j\leq n,\, 1\leq k\leq d]$ 
in $n^2d$ variables. Denote by
$$X_k=\left(\begin{array}{ccc}
x_{11}(k) & \cdots & x_{1n}(k)\\
\vdots & & \vdots \\
x_{n1}(k) & \cdots & x_{nn}(k)\\
\end{array}
\right)
$$%
the $k^{\rm th}\!$ {\it generic} matrix. The action of $GL(n)$ on $V$ induces the action on $R$ as follows: 
$$g\cdot x_{ij}(k)= (i,j)^{\rm th}\text{ entry of }g^{-1}X_k g$$%
for all $g\in GL(n)$.  The algebra of {\it $GL(n)$-invariants of matrices} is
$$R^{GL(n)}=\{f\in \mathbb{F}[V]\,|\,g\cdot f=f\;{\rm for\; all}\;g\in GL(n)\}.$$

Denote coefficients in the characteristic polynomial
of an $n\times n$ matrix $X$ by $\sigma_t(X)$, i.e., %
\begin{equation}\label{eq1_intro} 
\det(X+\lambda E)=\sum_{t=0}^{n} \lambda^{n-t}\sigma_t(X).
\end{equation}%
In particular, $\sigma_0(X)=1$, $\sigma_1(X)=\tr(X)$, and $\sigma_n(X)=\det(X)$. 

Given $a=x_{i_1}\cdots x_{i_r}\in\M$,  we set $X_{a}=X_{i_1}\cdots X_{i_r}$. It is known that the algebra $R^{GL(n)}\subset R$ is generated over $\FF$ by $\sigma_t(X_a)$, where $1\leq t\leq n$ and $a\in\M$ (see~\cite{Donkin92a}). Note that in the case of $p=0$ the algebra $R^{GL(n)}$ is generated by  $\tr(X_a)$, where $a\in\M$. Relations between the mentioned generators were established in~\cite{Zubkov96}. 

\begin{remark}\label{remark_G} If $G$ belongs to the list $O(n)$, $S\!p(n)$, $SO(n)$, $SL(n)$, then we can define the algebra of invariants $R^{G}$ in the same way as for $G=GL(n)$.  A generating set for the algebra $R^G$ is known, where we assume that $\Char{\mathbb{F}}\neq2$ in the case of $O(n)$ and $SO(n)$ (see~\cite{Zubkov99},~\cite{Lopatin_so_inv}). In case $p=0$ and $G\neq SO(n)$ relations between generators of $R^{G}$ were described in~\cite{Procesi76}. In case $p\neq2$ relations for $R^{O(n)}$ were described in~\cite{Lopatin_Orel},~\cite{Lopatin_free_rel}. 
\end{remark}
\bigskip


By the Hilbert--Nagata Theorem on invariants, $R^{GL(n)}$ is a finitely generated
$\NN_0$-graded algebra by degrees, where $\deg{\si_t(X_a)}=t\deg{a}$ for $a\in\M$. But the above mentioned generating set is not finite. In~\cite{Domokos02} the following finite generating set for $R^{GL(n)}$ was established:
\begin{enumerate}
\item[$\bullet$] $\si_t(X_a)$, where $1\leq t\leq \frac{n}{2}$, $a\in\M$, $\deg{a}\leq C_{n,d}$; 

\item[$\bullet$] $\si_t(X_i)$, where $\frac{n}{2}<t\leq n$, $1\leq i\leq d$.
\end{enumerate}
We obtain a smaller generating set.

\begin{theo}\label{theo_gen_system}
The algebra $R^{GL(n)}$ is generated by the following finite set:
\begin{enumerate}
\item[$\bullet$] $\si_t(X_a)$, where $t=1$ or $p\leq t\leq \frac{n}{2}$, $a\in\M$, $\deg{a}\leq C_{[n/t],d}$; 

\item[$\bullet$] $\si_t(X_i)$, where $\frac{n}{2}< t\leq n$, $p\leq t$, $1\leq i\leq d$.
\end{enumerate}
\end{theo}
\bigskip

To prove the theorem,  we need the following notions. Let $1\leq t\leq n$. For short, we write $\si_t(a)$ for $\si_t(X_a)$, where $a\in \M$.  Amitsur's formula~\cite{Amitsur_1980} enables us to consider $\si_t(a)$ with $a\in \M_{\FF}$ as an invariant from $R^{GL(n)}$ for all $t\in\NN$. Zubkov~\cite{Zubkov96} established that the ideal of relations for  $R^{GL(n)}$ is generated by $\si_t(a)=0$, where $t>n$ and $a\in\M_{\FF}$. More details can be found, for example, in~\cite{Lopatin_Orel}. Denote by $I(t)$ the $\FF$-span of elements $\si_{t_1}(a_1)\cdots \si_{t_r}(a_r)$, where $r>0$, $1\leq t_1,\ldots, t_r\leq t$, and $a_1,\ldots,a_r\in\M$. For short, we write $I$ for $I(n)=R^{GL(n)}$. Denote by $I^{+}$ the subalgebra generated by $\NN_0$-homogeneous elements of $I$ of positive degree. Obviously, the algebra $I$ is generated by a set $\{f_k\} \subset I$~if and only if $\{\ov{f_k}\}$ is a basis of $\ov{I}={I}/{(I^{+})^2}$. Given an $f\in I$, we write $f\equiv0$ if $\ov{f}=0$ in $\overline{I}$, i.e., $f$ is equal to a polynomial in elements of strictly lower degree.

\begin{proof} Let $1\leq t\leq n$, $m=[n/t]$, and $a,b\in\M_{\FF}$. We claim that
\begin{eq}\label{eq_claim2}
\text{there exists an }f\in I(t-1) \text{ such that }\si_t(ab^m)\equiv f.
\end{eq}%
To prove the claim we notice that the inequality $(m+1)t>n$ and the description of relations for $R^{GL(n)}$ imply $\si_{(m+1)t}(a+b)=0$.  Taking homogeneous component of degree $t$ with respect to $a$ and degree $mt$ with respect to $b$, we obtain that $\si_t(ab^m)\equiv0$ or $\si_t(ab^m)\equiv \sum_i \al_i\si_{t_i}(a_i)$, where $\al_i\in\FF^{\ast}$, $1\leq t_i<t$, and $a_i$ is a monomial in $a$ and $b$ for all $i$. By Amitsur's formula, $\si_{t_i}(a_i)\equiv\sum_j \be_{ij}\si_{r_{ij}}(b_{ij})$ for some $\be_{ij}\in\FF^{\ast}$, $1\leq r_{ij}\leq t_i$,  $b_{ij}\in\M$. Thus, $\sum_i \al_i\si_{t_i}(a_i)\in I(t-1)$ and the claim is proven. 

Consider a monomial $c\in\M$ satisfying $\deg{c}>C_{m,d}$. Then $c=c'x$ for some letter $x$ and $c'\in\M$. Since $c'=0$ in $N_{m,d}$, we have $c'=\sum_i\ga_i u_i v_i^m w_i$ for some $u_i,w_i\in\M_1$, $v_i\in\M_{\FF}$, $\ga_i\in\FF$. Thus 
$\si_t(c)=\si_t(\sum_i \al_i u_i v_i^m w_i x)$. Applying Amitsur's formula, we obtain that $\si_t(c)-\sum_i\al_i^t\si_t(u_i v_i^m w_i x)\in I(t-1)$. Statement~\Ref{eq_claim2} implies \begin{eq}\label{eq3}
\si_t(c)\equiv h \text{ for some }h\in I(t-1).
\end{eq}%
Consecutively applying~\Ref{eq3} to $t=n,n-1,\ldots,2$ we obtain that $R^{GL(n)}$ is generated by $\si_t(a)$, where $1\leq t\leq n$, $a\in\M$, $\deg{a}\leq C_{[n/t],d}$.  Note that if $t>\frac{n}{2}$, then $m=1$ and $C_{m,d}=1$. If $t<p\leq n$, then the Newton formulas imply that $\si_t(a)$ is a polynomial in $\tr(a^i)$, $i>0$ (the explicit expression can be found, for example, in Lemma~10 of~\cite{Lopatin_bplp}). The last two remarks complete the proof.
\end{proof}

\begin{conj}\label{conj_inv}
The algebra $R^{GL(n)}$ is generated by elements of degree less or equal to $C_{n,d}$.
\end{conj}

\begin{remark}
Theorem~\ref{theo_gen_system} and the inequality $C_{n,d}\geq n$ imply that  to prove Conjecture~\ref{conj_inv} it is enough to show that 
$$t\,C_{[n/t],d}\leq C_{n,d}$$
for all $t$ satisfying $p\leq t\leq \frac{n}{2}$. Thus it is not difficult to see that Conjecture~\ref{conj_inv} holds for $n\leq5$. Moreover, as it was proven in~\cite{Domokos02} (and also follows from Theorem~\ref{theo_gen_system}), Conjecture~\ref{conj_inv} holds in case $p=0$ or $p>\frac{n}{2}$.
\end{remark}

\section*{Acknowledgements}
This paper was supported by RFFI 10-01-00383a and FAPESP No.~2011/51047-1. The author is grateful for this support.

\end{document}